\documentclass[11pt]{amsart}

 \usepackage{graphicx,color,amssymb, amsmath, amsthm, mathrsfs}
\usepackage{geometry}
\usepackage[all]{xy}
\usepackage{amssymb}
\usepackage{slashed}
\usepackage{verbatim}
\usepackage{mathtools}
\usepackage[colorlinks]{hyperref}
\hypersetup{citecolor=blue}
\usepackage{todonotes}  % to use type:     \todo[inline]{your message goes here}
\usepackage{mathdots}
\usepackage{tikz}
\usetikzlibrary{matrix}

\swapnumbers
\newtheorem{theorem}[subsection]{Theorem}
\newtheorem{lemma}[subsection]{Lemma}
\newtheorem{proposition}[subsection]{Proposition}
\newtheorem{corollary}[subsection]{Corollary}
\theoremstyle{definition}
\newtheorem{definition}[subsection]{Definition}

\newtheorem{remark}[subsection]{Remark}
\newtheorem{example}[subsection]{Example}

\usepackage{calrsfs}

\usepackage[T1]{fontenc} 
\usepackage{textcomp}
\usepackage{times}
 \usepackage[scaled=0.92]{helvet} 

\renewcommand{\tilde}{\widetilde} %\wwtilde with mtpro
\newcommand{\End}{\textnormal{End}}
\newcommand{\C}{\mathbf{C}}

\newcommand{\R}{\mathbf{R}}
\newcommand{\Q}{\mathbf{Q}}
\newcommand{\Z}{\mathbf{Z}}

\renewcommand{\phi}{\varphi}    % Maybe search and replace for this after everything is finished
\renewcommand{\epsilon}{\varepsilon}    % Maybe search and replace for this after everything is finished
\renewcommand{\leq}{\leqslant }% Maybe search and replace for this after everything is finished

\newcommand{\bea}          {\begin{eqnarray}}
\newcommand{\eea}          {\end{eqnarray}}
\newcommand{\beastar}          {\begin{eqnarray*}}
\newcommand{\eeastar}          {\end{eqnarray*}}

\newcommand{\alg}{\textnormal{alg}}

\usepackage{amsmath,amssymb,amsfonts,amsthm}

\newcommand{\ch}{\textnormal{Ch}}
\newcommand{\Pic}{\textnormal{Pic}}
\DeclareFontFamily{OT1}{pzc}{}
\DeclareFontShape{OT1}{pzc}{m}{it}{<-> s * [1.20] pzcmi7t}{}
\DeclareMathAlphabet{\mathpzc}{OT1}{pzc}{m}{it}

\begin{document}
 
\date{\today}
\title{Hecke modules for arithmetic groups via bivariant $K$-theory}

\author[B. Mesland]{Bram Mesland}
\address{\normalfont{University of Bonn, Endenicher Allee 60, 53115, Germany}}
\email{mesland@math.uni-bonn.de}

\author[M.H. \c{S}eng\"un]{Mehmet Haluk \c{S}eng\"un}
\address{\normalfont{University of Sheffield,
School of Mathematics and Statistics, Hicks Building,
Sheffield S3 7RH }}
\email{m.sengun@sheffield.ac.uk}

\begin{abstract} 
\noindent Let $\Gamma$ be a lattice in a locally compact group $G$. In another work, we used $KK$-theory to equip with Hecke operators the $K$-groups of any $\Gamma$-$C^{*}$-algebra on which the commensurator of $\Gamma$ acts. When $\Gamma$ is arithmetic, this gives Hecke operators on the $K$-theory of certain $C^{*}$-algebras that are naturally associated with $\Gamma$. In this paper, we first study the topological $K$-theory of the arithmetic manifold associated to $\Gamma$. We prove that the Chern character commutes with Hecke operators. Afterwards, we show that the Shimura product of double cosets naturally corresponds to the Kasparov product and thus that the $KK$-groups associated to an arithmetic group $\Gamma$ become true Hecke modules. We conclude by discussing Hecke equivariant maps in $KK$-theory in great generality and apply this to the Borel-Serre compactification as well as various noncommutative compactifications associated with $\Gamma$. Along the way we discuss the relation between the $K$-theory and the integral cohomology of low-dimensional manifolds as Hecke modules.
\end{abstract}

\maketitle

\tableofcontents 
\section{Introduction}
Let $\Gamma$ be a lattice in a locally compact group $G$ with commensurator $C_G(\Gamma)$. Let $S \subset C_G(\Gamma)$ be a group containing $\Gamma$. In \cite{mesland-sengun}, for $g \in S$ and $B$ a $S$-$C^*$-algebra (that is, a $C^*$-algebra on which $S$ acts via automorphisms), we constructed elements $[T_g] \in KK_0(B \rtimes_r \Gamma, B \rtimes_r \Gamma)$. We introduced {\em analytic Hecke operators} on any  module over $KK_0(B \rtimes_{r} \Gamma, B \rtimes_{r} \Gamma)$ as the endomorphisms arising from the classes $[T_g]$. In the present paper we prove several structural results about these Hecke operators, showing that they generalise the well-known cohomological Hecke operators in a way that is compatible with the Chern character and the double-coset Hecke ring of Shimura. 

The double-coset Hecke ring of Shimura is well-known to number theorists. In the widely studied case where $\Gamma$ is an arithmetic group, the Hecke ring acts linearly on various spaces of automorphic forms associated to $\Gamma$, providing a rich supply of symmetries (\cite[Chapter 3]{Shimura}). Those automorphic forms that are simultaneous eigenvectors of these symmetries are conjectured, and proven in many cases, to have deep connections to arithmetic (\cite{clozel, taylor}). The Hecke ring also acts on the cohomology of the arithmetic manifold $M$ associated to $\Gamma$ and there is a Hecke equivariant isomorphism between spaces of automorphic form associated to $\Gamma$ and cohomology of $M$ twisted with suitable local systems (\cite{franke, Shimura}). The passage to cohomology leads to many fundamental results and new insights on the arithmetic of automorphic forms. The results of this paper, together with those of \cite{mesland-sengun}, offer an analytic habitat for the Hecke ring by providing ring homomorphisms from the Hecke ring to suitable $KK$-groups. The passage to $KK$-theory extends the scope of the action of the Hecke ring beyond cohomology and allows for the possibility of using of tools from operator $K$-theory in the study of automorphic forms. 

Let  us describe the results of the paper more precisely. In Section \ref{chern}, we consider the situation where $S$ acts on a locally compact Hausdorff space $X$. Assume that $\Gamma$ acts freely and properly on $X$ and put $M=\Gamma \backslash X$. It is well-known that the $C^*$-algebras $C_0(X) \rtimes_{r} \Gamma$ and $C_0(M)$ are Morita equivalent, so  $$KK_0(C_0(X) \rtimes_{r} \Gamma, C_0(X) \rtimes_{r} \Gamma) \simeq KK_{0}(C_{0}(M),C_{0}(M)),$$ and thus for any $g \in S$ we obtain a class $[T_g] \in KK_{0}(C_{0}(M),C_{0}(M))$. The element $g$ gives rise to a cover $M_{g}$ of $M$ and a pair of covering maps, forming the \emph{Hecke correspondence} $M\xleftarrow{s} M_{g} \xrightarrow{t} M$. In \cite{mesland-sengun} it was shown that the class $[T_g]$
 corresponds to the class of this Hecke correspondence, that is, 
 $$[T_g] = [M\leftarrow M_{g}\rightarrow M] \in KK_{0}(C_{0}(M),C_{0}(M)).$$

This class induces a Hecke operator $T_{g}:K^{*}(M)\to K^{*}(M)$ on topological $K$-theory. In this paper we show that the Chern character
\[\textnormal{Ch} : K^{0}(M)\oplus K^{1}(M)\to H^{\textnormal{ev}}(M,\Q)\oplus H^{\textnormal{odd}}(M,\Q),\]
is \emph{Hecke equivariant}. Here we equip $H^{*}(M,\Q)$ with Hecke operators in the usual way using the Hecke correspondence $M\xleftarrow{s} M_{g} \xrightarrow{t} M$ (see, for example, \cite{lee}). 

In Section \ref{bianchi}, we specialize to non-compact arithmetic hyperbolic 3-manifolds $M$. Let $\overline{M}$ be the Borel-Serre compactification of $M$. Consider the diagram 
\begin{equation} \label{pairing}
\xymatrix{ K^0(M) \ar[d] & \times & K_0(M)  \ar[r] & \Z \\ 
                     H^2(\overline{M}, \partial \overline{M}, \Z) & \times & H_2(\overline{M}, \partial \overline{M}, \Z) \ar[u] \ar[r]& \Z 
}
\end{equation} 
Here horizontal arrows are given by the standard pairings with respect to which the Hecke operators are adjoint. The vertical arrows are Hecke equivariant isomorphisms; we establish the one on the left via the results of  Section \ref{chern}  and the one on the right was proven in \cite{mesland-sengun}. Using the relative index theorem, we show that the diagram commutes. Using very different techniques, we proved a similar result in \cite{mesland-sengun} where the 
$K$-groups of $M$ were replaced with those of the reduced group $C^*$-algebra $C_r^*(\Gamma)$ of $\Gamma$.

In Section \ref{shimura} we prove the main result of the paper. The double-coset Hecke ring $\Z[\Gamma, S]$ is the free abelian group on the double cosets $\Gamma g\Gamma$, with $g\in S$, equipped with the  Shimura product (\cite{Shimura}). We show that the map $\Gamma g^{-1} \Gamma \mapsto [T_{g}]$ extends to a ring homomorphism
$$\Z[\Gamma, S]\rightarrow KK_0(B \rtimes_{r} \Gamma, B \rtimes_{r} \Gamma),$$
for any $S$-$C^{*}$-algebra $B$.
As mentioned in the second paragraph of this introduction, this homomorphism provides the Hecke ring $\Z[\Gamma, S]$ with a new habitat. The universality property of $KK$-theory \cite{higson} implies that for any additive functor $F$ on separable $C^*$-algebras that is homotopy invariant, split-exact and stable, the abelian groups $F(B\rtimes_{r}\Gamma)$ are modules over $\Z[\Gamma, S]$. For example, let $\Gamma$ be an arithmetic group in a semi-simple real Lie group $G$. By taking $F$ to be local cyclic cohomology and $B=C_0(X)$ where $X$ is the symmetric space of $G$, we recover the action of the Hecke ring on the cohomology of the arithmetic manifold $X / \Gamma$. In \cite{mesland-sengun}, we took $F$ to be $K$-homology and worked with three different $S$-$C^*$-algebras $B$ that were naturally associated to $\Gamma$.   

In Section \ref{exact}, we show that a $\Gamma$-exact and $S$-equivariant extension,
\[0\to B\to E\to A\to 0,\]
of $C^{*}$-algebras induces Hecke equivariant long exact sequences relating the $KK$-groups of the crossed products
$B\rtimes_{r}\Gamma, E\rtimes_{r}\Gamma$ and $A\rtimes_{r}\Gamma$. In particular, suppose that $X$ is a free and proper $\Gamma$-space on which $S$ acts by homeomorphisms, $\overline{X}$ a partial $S$-compactification of $X$ with boundary $\partial X:=\overline{X}\setminus X$. Then the extension 
$$0\to C_{0}(X)\to C_0(\overline{X})\to C_0(\partial X)\to 0,$$
induces a Hecke equivariant exact sequence
$$\xymatrix{K_1(C_{0}(X)\rtimes_{r}\Gamma) \ar[r] & K_1(C_{0}(\overline{X})\rtimes_{r}\Gamma) \ar[r] & K_1(C_{0}(\partial X)\rtimes_{r}\Gamma) \ar[d] \\ 
K_0(C_{0}(\partial X)\rtimes_{r}\Gamma ) \ar[u] & K_0(C_{0}(\overline{X})\rtimes_{r}\Gamma) \ar[l] & K_0(C_{0}(X)\rtimes_{r}\Gamma), \ar[l]}$$
of $\mathbf{Z}[\Gamma,S]$-modules. The results of Sections \ref{shimura} and \ref{exact} hold for the full crossed product algebras as well.  

Let ${\bf G}$ be a reductive algebraic group and $\Gamma\subset {\bf G}(\mathbf{Q})$ an arithmetic group. Then the Borel-Serre partial compactification $\overline{X}$  of the associated global symmetric space $X$ is a proper ${\bf G}(\Q)$-compactification. The associated Morita equivalences provide a Hecke equivariant isomorphism of above six-term exact sequence with the topological $K$-theory exact sequence of the Borel-Serre compactification of the arithmetic manifold $X/ \Gamma$ and its boundary.

The generality of our methods allow to also consider various noncommutative compactifications. One family of examples are  the Hecke equivariant Gysin exact sequences studied in \cite{mesland-sengun} coming from the geodesic compactification of hyperbolic $n$-space. Other examples of interest come from the Floyd boundary of $\Gamma$, such as the the boundary of tree associated to $SL(2,\mathbf{Z})$ and the Bruhat-Tits building of a $p$-adic group and its boundary. In most of these cases not all of the crossed products are Morita equivalent to a commutative $C^{*}$-algebra. 

\subsection*{Acknowledgements} We gratefully acknowledge our debt to Heath Emerson for suggesting Theorem \ref{KKequiv},  to John Greenlees and Dimitar Kodjabachev for their help with stable homotopy theory, to Matthias Lesch for a discussion on relative index theory and to Paul Mitchener for a discussion on the universal property of $KK$-theory. Finally we thank the anonymous referee for helpful suggestions.

\subsection*{Set-up and notation} \label{set-up}

The following set-up will hold for the whole paper. Let $G$ be a locally compact group and  $\Gamma \subset G$ a torsion-free discrete subgroup. Recall that two subgroups $H,K$ of $G$ are called {\em commensurable} if $H \cap K$ is of finite index in both $H$ and $K$. The commensurator $C_G(\Gamma)$ of $\Gamma$ (in $G$) is the group of elements $g \in G$ for which $\Gamma$ and $g\Gamma g^{-1}$ are commensurable. Moreover $S$ will denote a subgroup of $C_G(\Gamma)$ containing $\Gamma$.

%%%%%%%%%%%%%%%%%%%%%%%%%%%%%%%
\section{Hecke equivariance of the Chern character} \label{chern}

In this section, we shall assume that $S$ acts on a locally compact Hausdorff space $X$ and that the action of $\Gamma$ on $X$ is free and proper. Let $M$ denote the Hausdorff space $X / \Gamma$. Given an element $g \in S$, we put $M_g := X / \Gamma_g$ and 
$M^g := X / \Gamma^g$ where $\Gamma^g:=\Gamma \cap g^{-1}\Gamma g$ and $\Gamma_g:=\Gamma \cap g\Gamma g^{-1} = g\Gamma^g g^{-1}$. Note that $s : M_g \rightarrow M$ and $s': M^g \rightarrow M$ are finite sheeted covers (of the same degree) and the map $c: M_g \to M^g$ defined by $x\Gamma_g \mapsto g^{-1} x \Gamma^g$ is a homeomorphism. We obtain a second finite covering $t:= s' \circ c : M_g \to M$.

We shall equip the topological $K$-theory of $M$ with Hecke operators via two different constructions, one analytical, arising from a $KK$-class and the other topological, arising from a correspondence. We will see that these two constructions give rise to the same Hecke operator. Afterwards, we will show that the Chern character between the $K$-theory and the ordinary cohomology of $M$ is Hecke equivariant.

\subsection{Analytic Hecke operators.} \label{analytic-hecke} Let $g \in S$. As mentioned in the Introduction, thanks to a Morita equivalence, the analytically constructed class $[T_g] \in KK_0(C_0(X) \rtimes \Gamma, C_0(X) \rtimes \Gamma)$ gives rise to a class $[T_g^M] \in KK_0(C_0(M), C_0(M))$. This latter class has a simpler description which we now recall.  

The conditional expectation  
\[\rho:C_{0}(M_{g})\to C_{0}(M),\quad \rho(\psi)(m)=\sum_{x\in t^{-1}(m)} \psi(x), \] 
and right module structure
\[ \psi\cdot f(x):=\psi(x)f(t(x))\]
give $C_{0}(M_{g})$ a right $C_{0}(M)$-module which we will denote by $T_{g}^{M}$.  Because the map $s: M_{g}\to M$ is proper, there is a left action of $C_0(M)$ on $T_g^M$ by compact operators
\[C_{0}(M)\to \mathbb{K}(T_{g}^{M}),\quad f\cdot\psi(x)=f(s(x))\psi(x).\]
Then $[T_{g}^{M}]\in KK_{0}(C_{0}(M),C_{0}(M))$ is the class of this bimodule. 

Observe that $M\xleftarrow{s} M_{g}\xrightarrow{t}M$ defines a {\bf correspondence} in the sense of \cite{connes-skandalis}. Associated to this correspondence, there exists a class 
$[s_*] \otimes [t!] \in KK_0(C_0(M), C_0(M))$ where $t!$ is the wrong-way cycle arising from $t$. As $t$ is simply a finite covering of manifolds, it follows from \cite[Prop. 2.9]{connes-skandalis} 
that $t!$ acquires a simpler description and it is then not hard to see that $[s_*] \otimes [t!]$ equals $[T_{g}^{M}]$ above. 

\begin{definition}\label{manifoldHecke} Let $M=X/\Gamma$ as above. For any separable $C^{*}$-algebra $C$, the \emph{analytic Hecke operators}
\[T_{g}:KK_{*}(C_{0}(M),C)\to KK_{*}(C_{0}(M),C), \] 
\[T_{g}:KK_{*}(C, C_{0}(M))\to KK_{*}(C,C_{0}(M)) ,\]
are defined to be the Kasparov product with the class $[T_{g}^{M}]\in KK_{0}(C_{0}(M),C_{0}(M))$.
\end{definition}

An important case is when one takes $C\simeq \C$. Then we obtain analytic Hecke operators on the topological $K$-theory of $M$:
$$T_g : K^*(M) \rightarrow K^*(M).$$

\subsection{Topological Hecke operators.} We now proceed to give an ``elementary'' description of our Hecke operators in the special case of topological $K$-theory. To do this, we will follow the description of Hecke operators on ordinary cohomology from correspondences (see, for example, \cite{mesland-sengun}). To this end, we shall introduce the ``transfer map'' machinery from stable homotopy theory which will allow us to deal with generalized cohomology theories at no extra cost.

To a finite covering map $p: (Y,B) \rightarrow (X,A)$ of pairs of spaces (that is, a finite covering $p:Y \to X$ with subspaces $A\subset X$ and $B \subset Y$ such that $B=p^{-1}(A)$), there is a well-known construction (see \cite[Construction 4.1.1, Thm. 4.2.3]{adams}) and \cite{kahn-priddy}) that associates to $p$ a map of suspension spectra $p^! : \Sigma^{\infty} (X/A) \rightarrow \Sigma^{\infty} (Y/B)$. Via pre-composition with $p^!$, for any generalized cohomology theory $h^*$ with spectrum $E$, we obtain a homomorphism called the {\bf transfer map}
$$p^! : h^n(Y,B)=[\Sigma^\infty S^n \wedge \Sigma^\infty (Y/B), E] \longrightarrow h^n(X,A)=[\Sigma^\infty S^n \wedge \Sigma^\infty (X/A), E].$$ 

This transfer map agrees with the usual one in the case of ordinary cohomology (see \cite[Props. 2.1]{kahn-priddy}). In the case of topological $K$-theory, 
the transfer map is induced by the direct image map of Atiyah \cite{atiyah} (see \cite[Props. 2.3]{kahn-priddy}). Recall that if $f: Y \rightarrow X$ is a finite covering map and $E \rightarrow X$ is a vector bundle, then the direct image bundle  
$f^! E \rightarrow Y$ has fiber $(f^! E)_y$ at $y \in Y$ given by the direct sum ${\displaystyle \bigoplus_{f(x)=y}} E_x$. 

\begin{definition} Given any generalized cohomology theory $h^*$ with spectrum $E$ and $g \in S$, the {\em topological Hecke operator} $T_g$ on $h^n(M)$ is defined as the composition
$$ h^n(M) \xrightarrow{s^*} h^n(M_g) \xrightarrow{t^!} h^n(M).$$
\end{definition}

In the case of topological $K$-theory, these topological Hecke operator agree with the analytic ones that we defined earlier.

\begin{proposition}Let $g \in S$. The analytic Hecke operator $T_g$ on $K^*(M)$ agrees with the topological Hecke operator $T_g$ on $K^*(M)$.
\end{proposition}

\begin{proof} Let us prove the statement for $K^0$ first. It suffices to show that, after we identify $K^0(M) \simeq K_0(C_0(M))$, the direct image map of Atiyah is induced by tensor product (from the right) with the $C_0(M)$-module $T_g^M$ defined above in Section \ref{analytic-hecke}. To that end, we need to show that for any vector bundle $E \to M_g$, there is a unitary isomorphism between the $C_0(M)$-modules of sections
\[\alpha: \Gamma(E)\otimes_{C_0(M_g)}C_0(M_g)_{C_0(M)}\xrightarrow{\sim} \Gamma(t^! E).\]
This is achieved by choosing an open cover $U_i$ of $M_g$ for which the covering map $t$ is homeomorphic. Let $\chi_i^{2}$ be a partition of unity subordinate to the $U_i$. Define
\[\alpha(\psi\otimes f)(m):=\left(\sum_i \chi_i(x)\psi(x)f(x)\right)_{x\in t^{-1}(m)}\in t^!E.\]
It is straightforward to check that this induces the desired unitary isomorphism. Note that the above is also observed in \cite[Lemma 3.12]{ramras}. 

To prove the claim for $K^1$, we will descend to $K^0$ and exploit, as we did above, the fact that transfer is implemented by the direct image map. Consider the diagram below.
\begin{small}
\begin{equation} \label{main}
\xymatrix{ K^1(M_g) \ar[d]^{t^!} \ar[r]^{\simeq \ \ \ \ \ } & K^0(M_g \times \R) \ar[d]^{(t \times {\rm Id})^!}  \\
                     K^1(M) \ar[r]^{\simeq \ \ \ \ \ }               & K^0(M \times \R)  }
\end{equation}
\end{small}
 
The vertical arrows are the transfer maps arising from the finite coverings $t: M_g \to M$ and $t \times {\rm Id} : M_g \times \R \to M \times \R$.  The horizontal isomorphisms follow from long exact sequences in topological $K$-theory associated to suitable pairs of spaces.  As the transfer map is natural and commutes with connecting morphisms (see \cite[p.123-124]{adams}), it follows that the diagram is commutative.

Note that $K^0(M\times \R) \simeq K_0(C_0(M) \otimes C_0(\R))$. Under the isomorphism 
$$KK_0(C_0(M), C_0(M)) \xrightarrow{\simeq} KK_{0}(C_0(M)\otimes C_0(\R),C_0(M)\otimes C_0(\R)),$$ 
our distinguished class $[T_g^M]$ gets sent to $[T_g^M \otimes C_0(\R)]$. Now the same argument as in the first paragraph of this proof shows that the direct image map of Atiyah, for the finite covering $M_g \times \R \xrightarrow{t \times {\rm Id}} M \times \R$, is induced by tensor product with the $C_0(M)\otimes C_0(\R)$-module $T_g^M \otimes C_0(\R)$.
\end{proof}

\subsection{} Given a pair of compact Hausdorff spaces $(X,A)$, we have the Chern character (see \cite[V.3.26]{karoubi})
$$\ch: K^i(X,A)  \longrightarrow PH^{i}(X,A,\Q), \quad i=0,1$$
where $PH^0$ (resp. $PH^1$) is the periodic cohomology group given by the direct sum of the even (resp. odd) degree ordinary cohomology groups.
The Chern character commutes with suspension and thus is a stable cohomology operation (of degree $0$). 

Now let $M$ be a non-compact arithmetic manifold. For $g \in C_G(M)$, let $\overline{M},\overline{M_g}$ denote the Borel-Serre compactification of $M,M_g$ respectively (see \cite{borel-serre} and also \cite[2.1.2]{mesland-sengun}). It is well-known that the finite covering maps $s,t : M_g \rightarrow M$ extend to finite coverings of pairs of spaces  $\bar{s},\bar{t} : (\overline{M_g}, \partial \overline{M_g}) \rightarrow (\overline{M}, \partial \overline{M})$. From these, we obtain Hecke operators $T_g$ on the relative groups $K^*(\overline{M}, \partial \overline{M})$ and $H^*(\overline{M}, \partial \overline{M},\Z)$. Notice that $K^*(\overline{M}, \partial \overline{M}) \simeq \tilde{K}^*(M^+) = K^*(M) \simeq K_*(C_0(M))$ where $M^+$ is the one-point compactification of $M$. Moreover, we have that 
$H^*(\overline{M}, \partial \overline{M},\Z) \simeq H^*_c(M,\Z)$ where $H^*_c$ denotes compactly supported cohomology. 

It follows that for a given arithmetic manifold $M$, by choosing $(X,A)=(M,\emptyset)$ if $M$ is compact and $(X,A)=(\overline{M}, \partial \overline{M})$ if $M$ is non-compact, we have the Chern character
$$\ch: K^i(M) \longrightarrow PH_c^i(M,\Q), \quad i=0,1$$
and both sides are Hecke modules. A most natural question is whether the Chern character commutes with the Hecke actions.
\begin{proposition} \label{chern-equiv} Let $M$ be an arithmetic manifold and $g \in C_G(M)$. The Chern character 
$$\ch: K^i(M)  \longrightarrow PH^{i}_c(M,\Q), \quad i=0,1$$
commutes with the action of the Hecke operator $T_g$ on both sides.
\end{proposition}
\begin{proof} 
Consider a cohomology operation $\Psi : E^*(\cdot) \rightarrow F^{*}(\cdot)$ of degree $0$ between two cohomology theories with spectra $E,F$. If $\Psi$ is stable, there is in fact a map of spectra $\Psi : E \rightarrow F$ and the the cohomology operation is simply the composition 
$$E^n(X,A)=[\Sigma^\infty S^n \wedge \Sigma^\infty (A/X), E] \longrightarrow F^{n}(X,A)=[\Sigma^\infty S^n \wedge \Sigma^\infty (X/A), F], \ \ f \mapsto \Psi \circ f.$$ 
It immediately follows that the transfer operator associated to a finite cover of pairs of spaces $p:(Y,B) \to (X,A)$ commutes with $\Psi$, that is, the following diagram commutes
$$
\xymatrix{
E^n(Y,B) \ar[r]^{\Psi} \ar[d]^{p^!} & F^n(Y,B) \ar[d]^{p^!} \\ 
E^n(X,A) \ar[r]^{\Psi}  & F^n(X,A). }
$$

Now let us go back to our setting. Let us first assume that $M$ is compact. Note that $H^*_c(M,\Z) = H^*(M,\Z)$ in this case. As it is a stable cohomology operation, the Chern character commutes with the natural map 
$s^*$ and also with the transfer map $t^!$, giving rise to the following commutative diagram:
$$ 
\xymatrix{ K^*(M) \ar[d]_{\ch} \ar[r]^{s^*} & K^*(M_g)  \ar[d]_{\ch} \ar[r]^{t^!} & K^*(M) \ar[d]_{\ch} \\ 
PH^*(M,\Q) \ar[r]^{s^*} & PH^*(M_g,\Q)  \ar[r]^{t^!} & PH^*(M,\Q) }
$$
showing that the Chern character map commutes with Hecke operators. 

For the case $M$ non-compact, the proof follows in the same way considering the diagram
$$ 
\xymatrix{ K^*(\overline{M}, \partial \overline{M}) \ar[d]_{\ch} \ar[r]^{\bar{s}^*} & K^*(\overline{M_g}, \partial \overline{M_g})  \ar[d]_{\ch} \ar[r]^{\bar{t}^{\ !}} & K^*(\overline{M}, \partial \overline{M}) \ar[d]_{\ch} \\ 
PH^*(\overline{M}, \partial \overline{M},\Q) \ar[r]^{\bar{s}^*} & PH^*(\overline{M_g}, \partial \overline{M_g},\Q)  \ar[r]^{\bar{t}^{\ !}} & PH^*(\overline{M}, \partial \overline{M},\Q) }
$$
where $\bar{s},\bar{t} : (\overline{M_g}, \partial \overline{M_g}) \rightarrow (\overline{M}, \partial \overline{M})$ are the extensions of $s,t : M_g \rightarrow M$ mentioned earlier.
\end{proof}

\begin{remark} The transfer map used above is an example of what is known as a {\em wrong way map}. 
In \cite[Remark 2.10(a)]{connes-skandalis}, Connes and Skandalis remark that given a $K$-oriented map $f: X \rightarrow Y$ between smooth manifolds, the wrong way maps $f^! : K(X) \rightarrow K(Y)$, induced by the Kasparov product with the class of the wrong way cycle $[f!]\in KK_{*}(C_{0}(X),C_{0}(Y))$, and $f^!: H_c(X,\Q) \rightarrow H_c(Y,\Q)$ commute under the Chern character {\it modulo} an error term $\textnormal{Td}(f)$ defined via the Todd genus of certain bundles that naturally arise. In our case, this error term vanishes and we get that the transfer map commutes with the Chern character as we proved above. 
\end{remark}

\begin{remark} Using the universal property of $KK$-theory, the Chern character can be obtained as the unique natural transformation
\[\textnormal{Ch}:KK_{*}(A,B)\to HL_{*}(A,B),\]
where $HL_{*}$ denotes bivariant local cyclic homology (see \cite{meyerbook, Puschnigg}). For a locally compact space $X$, the local cyclic homology of $C_{0}(X)$ recovers the compactly supported sheaf cohomology of $X$ (\cite[Theorem 11.7]{Puschnigg}). Thus ordinary cohomology admits an action of analytic Hecke operators via its structure as a module over $KK$-theory. It follows from the results of this section that the topological Hecke operators on ordinary cohomology arise from the analytic Hecke module structure.
\end{remark}

%%%%%%%%%
\section{Bianchi manifolds}  \label{bianchi}
In this section, we present a result about arithmetic non-compact hyperbolic 3-manifolds that complements the results obtained in Section $5$ of our previous paper \cite{mesland-sengun}. In that paper, for a Bianchi manifold $M$, we provided a Hecke equivariant isomorphism between $K_0(M)$ and 
$H_2(\overline{M}, \partial \overline{M},\Z)$ where $\overline{M}$ is the Borel-Serre compactification of $M$ (see \cite{borel-serre}). We show below that $H^2(\overline{M},\partial \overline{M},\Z)$ and $K^0(M)$ are isomorphic as Hecke modules and further argue that the cohomological pairing between $H^2$ and $H_2$ and the index pairing between $K^0$ and $K_0$ commute under these isomorphisms.

Let $\mathcal{O}$ be the ring of integers of an imaginary quadratic field and $\Gamma$ be a torsion-free finite index subgroup of the Bianchi group $\textrm{PSL}_2(\mathcal{O})$. Then 
$\Gamma$ acts freely and properly on the hyperbolic $3$-space $\mathbf{H}_3$. The associated hyperbolic $3$-manifold $M= \mathbf{H}_3 / \Gamma$ is known as a {\em Bianchi manifold}. It is well-known that any non-compact arithmetic hyperbolic $3$-manifold is commensurable with a Bianchi manifold.

\subsection{} \label{picard} For compact connected spaces $X$, denote by $\tilde{K}^{0}(X)$ the {\em reduced $K$-theory} of $X$, that is, the kernel of the map $K^{0}(X)\to \Z$ induced by $[E]\mapsto \textnormal{dim}_{\C}(E)$. Write $[n]\in K^{0}(X)$ the class of the trivial bundle $T^{n}$ of rank $n$ over $X$. For a vector bundle $E$, the top exterior power $\bigwedge^{\textnormal{dim} E} E$ is called the \emph{determinant line bundle} and denoted $\textnormal{det} \ E$. Let $\Pic(X)$ denote the {\em Picard group} of $X$, that is, the set of isomorphism classes of line bundles on $X$ together with the tensor product operation.

Let $M^+$ denote the one-point compactification of the Bianchi manifold $M$. Since $M^{+}$ is a CW-complex of dimension $3$, every complex vector bundle $E\to M^{+}$ splits as $E\simeq \det E\oplus T^{\textnormal{dim}_{\C}(E)-1}$ (see \cite[Corollary 4.4.1]{Weibel}). It follows from \cite[Corollary 2.6.2]{Weibel} that the map
\[\textnormal{dim}\oplus\det: K^{0}(M^{+})\to \mathbf{Z}\oplus \Pic(M^{+}),\quad E\mapsto (\textnormal{dim}_{\C}(E),[\det E]), \]\
is an isomorphism. Noting that $H^0(M^+,\Z) \simeq \Z$ and identifying $\Pic(M^+) \simeq H^2(M^+,\Z)$ via the first Chern class $c_1$, we obtain the isomorphism
$$ K^0(M^+) \to H^0(M^+,\Z) \oplus H^2(M^+,\Z)$$
induced by $[E] \mapsto \textnormal{dim}_\C(E) + c_1(\det E)$. Note that this map agrees with the Chern character since $E\simeq T^{\textnormal{dim}_{\C}(E)-1} \oplus \det E $ as mentioned above. By Prop. \ref{chern-equiv}, this isomorphism is Hecke equivariant.

Composing the Chern character with the projection map, we obtain a surjection $K^0(M^+) \rightarrow H^2(M^+,\Z)$ whose kernel is $\tilde{K}^0(M^+)=K^0(M)$. Noting that 
$H^2(M^+,\Z)$ is isomorphic to the compactly supported cohomology $H^2_c(M,\Z)$ which in turn is isomorphic to $H^2(\overline{M},\partial \overline{M},\Z)$, 
we obtain an isomorphism
\begin{equation} \label{chern-isom} K^0(M) \xrightarrow{\sim} H^2(\overline{M},\partial \overline{M},\Z) \end{equation}
that is Hecke equivariant.

\subsection{}
Given a line bundle $L\to \overline{M}$ and any connection $\nabla$ on $L$, let $F_\nabla=\textnormal{Tr}(\frac{-1}{2 \pi i}\nabla^{2})$ be the curvature 2-form of $\nabla$. Then it is well-known that $F_\nabla$ is closed and its image in $H^2(M,\R)$ is in fact integral and equals the first Chern class $c_1(L)$ of $L$. 

\begin{proposition} \label{relative} Let $(N,\partial N)\subset (\overline{M},\partial\overline{M})$ be an embedded surface, $L\to \overline{M}$ a line bundle that is trivial on $\partial \overline{M}$ and $\overline{N}$ be the closed subspace of $N$ obtained by removing an open neighborhood of $\partial N$ over which $L$ is trivial. View the interior $\mathring{N}$ of $N$ as a spin$^c$ surface with associated Dirac operator $\slashed{D}_{\mathring{N}}$ (see \cite[Section 5]{mesland-sengun}). We have
\begin{equation*}\langle [\slashed{D}_{\mathring{N}}], [L]-[1]\rangle=\int_{\overline{N}}F_\nabla\end{equation*}
for any connection $\nabla$ on $L$. Here $\langle \cdot , \cdot \rangle$ is the index pairing.
\end{proposition}
\begin{proof}It follows from the relative index theorem of \cite[Theorem 4.6]{Roerelativeindex} that
\begin{equation*}\langle [\slashed{D}_{\mathring{N}}], [L]-[1]\rangle=\int_{\overline{N}}\hat{A}(\mathring{N})\textnormal{Ch}(L|_{\mathring{N}})-\int_{\overline{N}}\hat{A}(\mathring{N}).\end{equation*}
Here $L|_{\mathring{N}}$ is the restriction of $L$ in to interior of $N$. Observe that ${\rm Ch}(L|_{\mathring{N}})= 1+ c_1(L|_{\mathring{N}}) = 1+ [F_\nabla |_{\mathring{N}}]$ where $\nabla$ is any chosen connection on $L$ and $F_\nabla |_{\mathring{N}}$ is the restriction of its curvature to $\mathring{N}$. The $\hat{A}$-genus $\hat{A}(\mathring{N})$ of $\mathring{N}$ equals $1$ as it only has nonzero components in forms of degree $0 \textnormal{ mod } 4$. The claim follows.
\end{proof}

The following is not necessary for the main result of this section, however we note it as it quickly follows from the above and \cite[Lemma 2.22]{BallmannBruning2}.
\begin{corollary}\label{fullints} If $N$ has finite volume, we have
$$\langle [\slashed{D}_{\mathring{N}}], [L]-[1]\rangle=\int_{\mathring{N}}F_\nabla,$$
for any connection $\nabla$ on $L$.
\end{corollary}

\begin{proposition} We have the equality
$$\langle [\slashed{D}_{\mathring{N}}], [L]-[1]\rangle=\langle [(N,\partial N)], c_{1}(L)\rangle$$ %=\textnormal{deg}(L|_{\mathring{N}}).$$
In particular, the isomorphisms 
$$K^0(M)\xrightarrow{\simeq} H^2(\overline{M},\partial\overline{M},\Z), \ \ \ \ \ \ K_0(M)\xleftarrow{\simeq} H_2(\overline{M},\partial\overline{M},\Z)$$ 
(see (\ref{chern-isom}) and \cite[Prop.5.6.]{mesland-sengun}) are compatible with the index pairing $$\langle - , - \rangle : K_0(M) \times K^0(M) \to \Z$$ 
and the integration pairing
$$\langle - , -  \rangle : H_2(\overline{M},\partial\overline{M},\Z) \times H^2(\overline{M},\partial\overline{M},\Z) \to \Z.$$
In other words, diagram (\ref{pairing}) of the Introduction is commutative.
\end{proposition}

\begin{proof}  It follows from our discussion in Section \ref{picard} that every element of $K^0(M)$ is of the form $[L]-[1]$ where $1$ is the trivial line bundle and $L \to M$ is a line bundle that is trivial at  infinity.  
Under isomorphism (\ref{chern-isom}), the image of $[L]-[1]$ is $c_1(L)$. Every class in $H_2(\overline{M},\partial\overline{M},\Z)$ is represented by a properly embedded surface 
$(N,\partial N) \subset (\overline{M}, \partial \overline{M})$ (see \cite[Section 5]{mesland-sengun}). Then the pairing $\langle [(N,\partial N)], c_{1}(L)\rangle$ is given by the integral $\int_N F_\nabla$ where 
$\nabla$ is any connection on $L$ and $F_\nabla$ is the associated curvature $2$-form as above. As $L$ is trivial at infinity, we can choose closed $\overline{N} \subset \mathring{N}$ so that $L$ is trivial outside $\overline{N}$ and it then follows that $\int_N F_\nabla = \int_{\overline{N}} F_\nabla$. Observe that the image of $[(N,\partial N)]$ in $K_0(M)$ under the isomorphism given in \cite[Prop.5.6.]{mesland-sengun}) is 
$[\slashed{D}_{\mathring{N}}]$. Now by Prop. \ref{relative}, we have the claim.
\end{proof}

%%%%%%%%%%%%%%%%%%%%%%%%%%%%%
\section{The double-coset Hecke ring and $KK$-theory}\label{shimura}
We recall the construction of the Hecke operators via $KK$-theory as put forward in \cite{mesland-sengun}. We then show that the multiplication of double-cosets corresponds to the Kasparov product of the associated $KK$-classes. 

\subsection{Bimodules over the reduced crossed product}\label{redcross}
For a $\Gamma$-$C^{*}$-algebra $B$, the \emph{reduced crossed product} $B\rtimes_{r}\Gamma$ is obtained as a completion of the convolution algebra $C_{c}(\Gamma, B)$ (see, for example, \cite{kasparovconspectus}). Let $g\in C_{G}(\Gamma)$ and $d:=[\Gamma:\Gamma^{g}]$. The double coset $\Gamma g^{-1}\Gamma$ admits a decomposition as a disjoint union 
\begin{equation}\label{decomp}\Gamma g^{-1}\Gamma=\bigsqcup_{i=1}^{d} g_i\Gamma, \quad g_{i}=\delta_i g^{-1},\quad \Gamma=\bigsqcup_{i=1}^{d}\delta_{i}\Gamma^{g},\end{equation}
where $\delta_{i}\in\Gamma$ form a complete set of coset representatives for $\Gamma^{g}$. We choose to work with $g^{-1}$ in order for our formulae to be in line with those in \cite{mesland-sengun}. Consider the elements $$t_{i}(\gamma)=t^{g}_{i}(\gamma):= g^{-1}_{\gamma(i)}\gamma g_{i}\in g\Gamma g^{-1},$$ where $i\mapsto \gamma(i)$ is induced by the permutation of the cosets in Equation \eqref{decomp}. From \cite[Lemma 2.3] {mesland-sengun} we recall the relations
$$ t_{i}(\gamma_{1}\gamma_{2})=t_{\gamma_{2}(i)}(\gamma_{1})t_{i}(\gamma_{2}),\quad t_{i}(\gamma^{-1})=t_{\gamma^{-1}(i)}(\gamma)^{-1},$$
which will be used in the sequel without further ado.\newline

Let $S\subset C_{G}(\Gamma)$ be a subgroup containing $\Gamma$ and $B$ an $S$-$C^{*}$-algebra. The free right $B\rtimes_{r}\Gamma$ module  $T^{\Gamma}_{g}\simeq (B\rtimes_{r}\Gamma)^{d}$ carries a left $B\rtimes_{r}\Gamma$ module structure given by
\begin{equation}\label{fullheckerep}(t_{g}(f)\Psi)_{i}(\delta)=\sum_{\gamma} g_{i}^{-1}f(\gamma)t_{i}(\gamma^{-1})^{-1}\Psi_{\gamma^{-1}(i)}(t_{i}(\gamma^{-1})\delta).\end{equation}
Equivalently, we have the covariant representation 
\begin{equation}\label{heckerep}(t_{g}(b)\cdot\Psi)_{i}(\delta):= g_{i}^{-1}(b)\Psi_{i}(\delta),\quad (t_{g}(u_{\gamma})\Psi)_{i}(\delta):= t_{i}(\gamma^{-1})^{-1}(\Psi_{\gamma^{-1}(i)}(t_{i}(\gamma^{-1})\delta )). \end{equation}
Details of the construction,  as well as the following definition, can be found in \cite[Section 2]{mesland-sengun}.

\begin{definition} Let $B$ be a separable $S$-$C^{*}$-algebra and $C$ a separable $C^{*}$-algebra. The \emph{Hecke operators}
\[T_{g}: KK_{*}(B\rtimes_{r}\Gamma,C)\to KK_{*}(B\rtimes_{r}\Gamma, C),\quad T_{g}: KK_{*}(C, B\rtimes_{r}\Gamma)\to KK_{*}(C, B\rtimes_{r}\Gamma).\]
are defined to be the Kasparov product with the class $[T^{\Gamma}_{g}]\in KK_{0}(B\rtimes_{r} \Gamma, B\rtimes_{r} \Gamma)$.
\end{definition}
We now give an equivalent description of the bimodules $T_{g}^{\Gamma}$. Consider the function space
\[C_c(\Gamma g^{-1} \Gamma, B)=\mathbf{C}[\Gamma g^{-1} \Gamma]\otimes^{\textnormal{alg}}_{\mathbf{C}}B.\]

The convolution product makes $C_{c}(\Gamma g^{-1} \Gamma, B)$ into a $C_{c}(\Gamma, B)$-bimodule:
\[f * \Psi(\xi)=\sum_{\gamma\in\Gamma}f(\gamma)\gamma(\Psi(\gamma^{-1}\xi)), \quad \Psi * f(\xi):=\sum _{\gamma\in\Gamma} \Psi(\xi\gamma)\xi f(\gamma^{-1}), \quad \xi\in \Gamma g^{-1}\Gamma.\]
Moreover we define the inner product
\begin{equation}\label{cosetrepinnerprod}
\langle\Phi,\Psi\rangle(\delta):=\sum_{\xi\in\Gamma g^{-1}\Gamma}\xi^{-1}(\Phi(\xi)^{*}\Psi(\xi\delta)),
\end{equation}
which makes $C_c(\Gamma g^{-1} \Gamma, B)$ into a pre-Hilbert-$C^{*}$-bimodule over $C_{c}(\Gamma,B)$.
\begin{lemma}\label{cosetHecke} For $g\in S\subset  C_{G}(\Gamma)$ the map 
\[\alpha: C_{c}(\Gamma g^{-1} \Gamma,B)\to C_{c}(\Gamma,B)^{d}\subset T^{\Gamma}_{g}, \quad \alpha(\Psi)_{i}(\delta):= g_{i}^{-1}\Psi (g_{i}\delta),\]
induces a unitary isomorphism of $B\rtimes_{r}\Gamma$-bimodules.
\end{lemma}
\begin{proof} The decomposition \eqref{decomp} shows that the map $\alpha$ has dense range.
Moreover $\alpha$ preserves the inner product:
\begin{align*}
\langle &\alpha(\Psi),\alpha(\Phi)\rangle (\delta)=\sum_{i}\alpha(\Psi)_{i}^{*}\alpha(\Phi)_{i}(\delta)=\sum_{i}\sum_{\gamma}\alpha(\Psi)_{i}^{*}(\gamma)\gamma\alpha(\Phi)_{i}(\gamma^{-1}\delta) \\
&=\sum_{i}\sum_{\gamma}\gamma(\alpha(\Psi)_{i}(\gamma^{-1})^{*}\alpha(\Phi)_{i}(\gamma^{-1}\delta))=\sum_{i}\sum_{\gamma}\gamma g_{i}^{-1}(\Psi(g_{i}\gamma^{-1})^{*}\Phi(g_{i}\gamma^{-1}\delta))\\
&=\sum_{\xi\in\Gamma g^{-1}\Gamma}\xi^{-1}(\Phi(\xi)^{*}\Psi(\xi\delta))=\langle \Psi,\Phi\rangle(\delta),
\end{align*}
from which it follows that $\alpha$ induces a unitary isomorphism on the $C^{*}$-module completions, which is in particular a right module map.

For the left module structure we compute
\begin{align}
\nonumber \alpha(f*\Psi)_{i}(\delta)&=g_{i}^{-1}(\sum_{\gamma\in\Gamma}f(\gamma)\gamma\Psi(\gamma^{-1}g_{i}\delta))=\sum_{\gamma\in\Gamma}g_{i}^{-1}f(\gamma)g_{i}^{-1}\gamma \Psi(g_{\gamma^{-1}(i)}t_{i}(\gamma^{-1})\delta))\\
\nonumber &=\sum_{\gamma\in\Gamma}g_{i}^{-1}f(\gamma)t_{i}(\gamma^{-1})^{-1}g_{\gamma^{-1}(i)}^{-1}\Psi(g_{\gamma^{-1}(i)}t_{i}(\gamma^{-1})\delta)\\ &=\sum_{\gamma\in\Gamma}g_{i}^{-1}f(\gamma)t_{i}(\gamma^{-1})^{-1}\alpha(\Psi)_{\gamma^{-1}(i)}(t_{i}(\gamma^{-1})\delta)=(t_{g}(f))(\alpha\Psi)_{i}(\delta)\label{lefthecke},
\end{align}

and we are done.
\end{proof}
Thus, the bimodules implementing the Hecke operators are completions of the $B$-valued functions on the associated double coset.

%%%%%
\subsection{The double-coset Hecke ring}
Let $S$ be a subgroup of $C_G(\Gamma)$ that contains $\Gamma$.
Following Shimura, we define the \emph{Hecke ring} $\mathbf{Z}[\Gamma, S]$ as the free abelian group on the double cosets $\Gamma g\Gamma$ with $g\in S$, equipped with the product
\begin{equation}\label{Shimproddef}  [\Gamma g^{-1}\Gamma]\cdot [\Gamma h^{-1}\Gamma]:=\sum_{k=1}^{K}m_{k}[\Gamma g_{i(k)}h_{j(k)}\Gamma],\end{equation}
where we have fixed finite sets $I$ and $J$ and coset representatives $\{g_{i}:i\in I\}$ and $\{h_{j}:j\in J\}$ for $\Gamma^{g}$ and $\Gamma^{h}$ in $\Gamma$, respectively. Moreover $m_{k}, i(k)$ and $j(k)$ are such that 
\begin{equation}\label{biactiondecomp} m_{k}:=\#\{(i,j): g_{i}h_{j}\Gamma=g_{i(k)}h_{j(k)}\Gamma\}\quad\textnormal{and}\quad \Gamma g^{-1} \Gamma h^{-1}\Gamma=\bigsqcup_{k=1}^{K}\Gamma g_{i(k)}h_{j(k)}\Gamma,\end{equation}
is a disjoint union. For well-definedness and other details of the construction we refer to \cite[Chapter 3]{Shimura}. We wish to show that, for an arbitrary $S$-$C^{*}$-algebra $B$, the map 
\begin{equation}T:\mathbf{Z}[\Gamma, S]\to KK_{0}(B\rtimes_{r}\Gamma, B\rtimes_{r}\Gamma),\quad [\Gamma g^{-1}\Gamma ]\mapsto T^{\Gamma}_{g},
\end{equation}
is a ring homomorphism. To this end, we introduce the following notions. By a \emph{bi}-$\Gamma$-\emph{set} we mean a set $V$ that carries both a left- and a right $\Gamma$-action, and the actions commute in the sense that for all $\gamma,\delta\in\Gamma$ and $v\in V$ we have $\gamma(v\delta)=(\gamma v)\delta$. 

The $\Gamma$-\emph{product} of a pair $(V,W)$ of bi-$\Gamma$-sets is the quotient of the Cartesian product $V\times W$ by the equivalence relation
\[(v,w)\sim (v',w')\Leftrightarrow \exists\gamma\in\Gamma\quad  v'=v\gamma,\quad w'=\gamma^{-1}w,\]
and is denoted by $V\times_{\Gamma} W$. The equivalence class of the pair $(v,w)$ is denoted $[v,w]$. The $\Gamma$-product is a bi-$\Gamma$-set via the induced left- and right $\Gamma$-actions
\[[v,w]\gamma:=[v,w\gamma],\quad \gamma [v,w]:=[\gamma v, w].\]
Let $\Gamma\subset S\subset C_{G}(\Gamma)$ be a subgroup and $V$ a bi-$\Gamma$-set. We say that $V$ is \emph{anchored in} $S$ if there is given a map $m:V\to S$ such that $m(\gamma v\delta)=\gamma m(v)\delta$ for all $v\in V$, $\gamma, \delta\in\Gamma$. We refer to $m$ as the \emph{anchor}. Of course any double coset $\Gamma g\Gamma$ with $g\in S$ is anchored in $S$ via the inclusion map.
\begin{lemma} Let $V$ and $W$ be bi-$\Gamma$-sets with anchor maps $m_{V}:V\to S, m_{W}:W\to S$. Then their $\Gamma$-product $V\times_{\Gamma}W$
is anchored in $S$ via the \emph{product anchor} $[v,w]\mapsto m_{V}(v)m_{W}(w)$.
\end{lemma}
The proof of this is straightforward. Note that if $V$ and $W$ are double $\Gamma$-cosets in $S$, anchored via their embeddings into $S$, then the product anchor of $V\times_{\Gamma}W$ need not be injective.\newline

We wish to relate the anchored bi-$\Gamma$-sets $\Gamma g^{-1}\Gamma\times_{\Gamma} \Gamma h^{-1}\Gamma$ and $\bigsqcup_{k=1}^{K}\bigsqcup_{\ell=1}^{m_{k}}\Gamma z_{k}\Gamma$. By virtue of Equation \eqref{biactiondecomp}  
we fix, once and for all, for each $z_{k}$ and $1\leq \ell \leq m_{k}$  a choice of 
%$i(k),j(k)$ such that $z_{k}:=g_{i(k)}h_{j(k)}$ satisfy \eqref{biactiondecomp}. 
%By virtue of Equation \eqref{biactiondecomp} we can set $i(k,1)=i(k), j(k,1)=j(k)$ and 
\emph{distinct} indices $i(k,\ell),j(k,\ell)$
such that $z_{k}\Gamma=g_{i(k,\ell)}h_{j(k,\ell)}\Gamma$. We thus write $z_{(k,\ell)}=g_{i(k,\ell)}h_{j(k,\ell)}$. 
Consider the left action of $\Gamma$ on the finite set $I\times J$ given by
\begin{equation}\label{labelaction}\gamma(i,j):=(\gamma(i),t^{g}_{i}(\gamma)(j)).\end{equation}

\begin{lemma}\label{bieq}With the above choices, the map 
\begin{align*}\omega: \bigsqcup_{k=1}^{K}\bigsqcup_{\ell=1}^{m_{k}}\Gamma z_{(k,\ell)}\Gamma &\to \Gamma g^{-1}\Gamma\times_{\Gamma} \Gamma h^{-1}\Gamma\\
\gamma z_{(k,\ell)}\delta &\mapsto [\gamma g_{i(k,\ell)},h_{j(k,\ell)}\delta],\end{align*}
where $i=i(k,\ell), j=j(k,\ell)$, is a $\Gamma$-bi-equivariant bijection of $S$-anchored bi-$\Gamma$-sets.
\end{lemma}

\begin{proof} By construction, $\omega$ is $\Gamma$-bi-equivariant and respects the anchors. We need only show that it is bijective.
This is achieved as follows: for each $k$ choose
\[\gamma^{k}_{1}=1,\gamma^{k}_{2},\cdots, \gamma^{k}_{d_{k}}\in \Gamma,\quad \textnormal{with } \Gamma z_{k}\Gamma=\bigsqcup_{n=1}^{d_{k}}\gamma^{k}_{n}z_{k}\Gamma.\]
We thus have
\begin{equation}\label{finegrain}\bigsqcup_{k=1}^{K}\bigsqcup_{\ell=1}^{m_{k}}\Gamma z_{(k,\ell)}\Gamma=\bigsqcup_{k=1}^{K}\bigsqcup_{\ell=1}^{m_{k}}\bigsqcup_{n=1}^{d_{k}}\gamma^{k}_{n}g_{i(k,\ell)}h_{j(k,\ell)}\Gamma.\end{equation}

The identities
\[[g_{i}\gamma, h_{j}\delta]=[g_{i},\gamma h_{j}\delta]=[g_i, h_{\gamma(j)} t^{h}_{j}(\gamma)\delta],\]
show that every element in the $\Gamma$-product $\Gamma g^{-1}\Gamma\times_{\Gamma} \Gamma h^{-1}\Gamma$ has a representative of the form $[g_i, h_j\gamma]$ and such representatives are unique because $g_i$ and $h_j$ form a complete set of coset representatives. We so obtain a set bijection 
\[\Gamma g^{-1}\Gamma\times_{\Gamma} \Gamma h^{-1}\Gamma\to \bigsqcup_{(i,j)\in I\times J} \{g_{i}\}\times h_{j}\Gamma,\quad [g_{i}\gamma, h_{j}\delta]\mapsto [g_i, h_{\gamma(j)} t^{j}_{h}(\gamma)\delta]. \]
It follows that $\omega$ restricts to bijections $$\omega: \gamma^{k}_{n}g_{i(k,\ell)}h_{j(k,\ell)}\Gamma \rightarrow \{g_{\gamma^{k}_{n}(i(k,\ell))}\}\times h_{t^{g}_{i}(\gamma^{k}_{n})(j(k,\ell))}\Gamma .$$ Therefore it suffices to show that the map
\begin{align*} N\times K \times L &\rightarrow I\times J \\
(n,k,\ell)&\mapsto \gamma^{k}_{n}(i(k,\ell),j(k,\ell)),\end{align*}
is bijective. By \cite[Proposition 3.2]{Shimura} it holds that
$$\sum_{k=1}^{K}m_{k}d_{k}=|I||J|=|I\times J|,$$ 
and thus 
we need only show that this map is injective, and then use a counting argument to obtain surjectivity. To this end we will prove that the equality
\begin{equation}\label{actionequal}\gamma^{k}_{n}(i(k,\ell),j(k,\ell))=\gamma^{k'}_{n'}(i(k',\ell'),j(k',\ell'))\end{equation}
implies that $(n,k,\ell)=(n',k',\ell')$. By \eqref{labelaction}, \eqref{actionequal} implies that
\[\gamma^{k}_{n}g_{i(k,\ell)}h_{j(k,\ell)}\Gamma=\gamma^{k'}_{n'}g_{i(k',\ell')}h_{j(k',\ell')}\Gamma\quad\textnormal{and thus}\quad \Gamma g_{i(k,\ell)}h_{j(k,\ell)}\Gamma=\Gamma g_{i(k',\ell')}h_{j(k',\ell')}\Gamma. \]
This in turn implies that $k=k'$ and thus $\gamma_{n}^{k}z_{k}\Gamma=\gamma^{k}_{n'}z_{k}\Gamma,$
so it folllows that $n=n'$. Lastly, we are left with $\gamma^{k}_{n}(i(k,\ell))=\gamma^{k}_{n}(i(k,\ell'))$, so $i(k,\ell)=i(k,\ell')$ which by construction implies that $\ell=\ell'$. This shows that the map $(n,k,\ell)\mapsto \gamma^{k}_{n}(i(k,\ell),j(k,\ell))$ is injective.
\end{proof}
Now let $V$ be a $\Gamma$-set with anchor $m:V\to S$ and $X$ a $S$-$(A,B)$-bimodule. We always consider $V$ as a discrete set.
We equip $C_{c}(V, X)$ with a $C_{c}(\Gamma,B)$ valued inner product via
\begin{align*}\langle \Phi,\Psi\rangle(\delta):=\sum_{v\in V}m(v)^{-1}\langle \Phi(v),\Psi(v\delta)\rangle.
\end{align*}
and left and right module structures via the $\Gamma$-action:
\begin{align*}&f *\Psi(v):=\sum_{\gamma}f(\gamma)\gamma\Psi(\gamma^{-1}v),\quad \Psi*f(v):=\sum_{\gamma}\Psi(v\gamma)m(v\gamma) f(\gamma^{-1}).
\end{align*}
Thus the completion gives a $C^{*}$-$(A\rtimes_{r}\Gamma,B\rtimes_{r}\Gamma)$-bimodule. Note that if $u:X\to Y$ is an $S$-equivariant unitary bimodule isomorphism and $\omega:W\to V$ an isomorphism  of $S$-anchored bi-$\Gamma$-sets, then
\[C_{c}(V,X)\to C_{c}(W,Y),\quad \Psi\mapsto u\circ\Psi\circ\omega,\]
is a unitary bimodule isomorphism.

By Lemma \ref{cosetHecke}, the bimodule $T^{\Gamma}_{g}$ for $g\in S$ is isomorphic to the completion of $C_{c}(\Gamma g^{-1}\Gamma,B)$ with anchor $m:\Gamma g^{-1}\Gamma \to S$ the set inclusion, and is thus a special case of the above construction. The formalism of anchored bi-$\Gamma$-sets allows for an elegant description of tensor products of their associated modules.
\begin{proposition}\label{GodOfTheSun} Let $S\subset C_{G}(\Gamma)$ be a subgroup and $A,B$ and $C$ be $S$-$C^{*}$-algebras. Let $V,W$ be $S$-anchored bi-$\Gamma$-sets, $X$ an $(A,B)$-$S$-bimodule, $Y$ a $(B,C)$-$S$-bimodule. Then the map
\begin{align*}\alpha: C_{c}(V,X)\otimes_{C_{c}(\Gamma,B)}^{\alg}C_{c}(W,Y)&\rightarrow C_{c}(V\times_{\Gamma}W,X\otimes_{B}Y),
\end{align*}
given by
\[\alpha(\Phi\otimes\Psi)[v,w]:=\sum_{\gamma}\Phi(v\gamma)\otimes m(v)\gamma\Psi(\gamma^{-1}w),\]
is an inner product preserving map of $(C_{c}(\Gamma, A),C_{c}(\Gamma,C))$-bimodules with dense range. Consequently their respective $C^{*}$-module completions are unitarily isomorphic $(A\times_{r}\Gamma,C\rtimes_{r}\Gamma)$-bimodules.
\end{proposition}
\begin{proof} The following calculation shows that $\alpha$ is unitary:
\begin{align*}\langle \alpha & (\Phi\otimes\Psi),\alpha(\Phi\otimes\Psi)\rangle(\delta)=\sum_{[v,w]} m(w)^{-1}m(v)^{-1}\langle\alpha(\Phi\otimes\Psi)(v,w),\alpha(\Phi\otimes\Psi)(v,w\delta) \rangle\\
&=\sum_{[v,w]} \sum_{\gamma,\varepsilon}m(w)^{-1}m(v)^{-1}\langle m(v)\gamma\Psi(\gamma^{-1}w),\langle \Phi(v\gamma),\Phi(v\varepsilon)\rangle m(v)\varepsilon\Psi(\varepsilon^{-1}w\delta) \rangle\\
&=\sum_{[v,w]} \sum_{\gamma,\varepsilon}m(w)^{-1}\langle \gamma\Psi(\gamma^{-1}w),m(v)^{-1}(\langle \Phi(v\gamma),\Phi(v\varepsilon)\rangle) \varepsilon\Psi(\varepsilon^{-1}w\delta) \rangle\\
&=\sum_{[v,w]} \sum_{\gamma,\varepsilon}m(\gamma^{-1}w)^{-1}\langle \Psi(\gamma^{-1}w),m(v\gamma)^{-1}(\langle \Phi(v\gamma),\Phi(v\varepsilon)\rangle) \gamma^{-1}\varepsilon\Psi(\varepsilon^{-1}w\delta) \rangle\\
%&=\sum_{(v,w)} \sum_{\gamma,\varepsilon}m(\gamma^{-1}w)^{-1}\langle \Psi(\gamma^{-1}w),m(v\gamma)^{-1}(\langle \Phi(v\gamma),\Phi(v\gamma\gamma^{-1}\varepsilon)%\rangle) \gamma^{-1}\varepsilon\Psi(\varepsilon^{-1}\gamma\gamma^{-1}w\delta) \rangle\\
&=\sum_{[v,w]} \sum_{\gamma,\varepsilon}m(\gamma^{-1}w)^{-1}\langle \Psi(\gamma^{-1}w),m(v\gamma)^{-1}(\langle \Phi(v\gamma),\Phi(v\gamma\varepsilon)\rangle) \varepsilon\Psi(\varepsilon^{-1}\gamma^{-1}w\delta) \rangle,\end{align*}
and by virtue of the equivalence relation on $V\times W$ we can replace the sum over equivalence classes $[v,w]\in V\times_{\Gamma}W$ and elements $\gamma\in\Gamma$ by a sum over $(v,w)\in V\times W$, and continue the calculation:
\begin{align*}
&=\sum_{v\in V}\sum_{w\in W} \sum_{\varepsilon}m(w)^{-1}\langle \Psi(w),m(v)^{-1}(\langle \Phi(v),\Phi(v\varepsilon)\rangle) \varepsilon\Psi(\varepsilon^{-1} w\delta) \rangle\\
&=\sum_{w} \sum_{\varepsilon}m(w)^{-1}\langle \Psi(w),\langle \Phi,\Phi\rangle (\varepsilon) \varepsilon\Psi(\varepsilon^{-1} w\delta) \rangle\\
&=\sum_{w} m(w)^{-1}\langle \Psi(w),\langle \Phi,\Phi\rangle *\Psi(w\delta) \rangle\\
&=\langle \Psi,\langle\Phi,\Phi\rangle\Psi\rangle (\delta).
\end{align*}
It is straightforward to establish that $\alpha$ is a bimodule map:
\begin{align*}\alpha(f*\Phi\otimes \Psi)[v,w]&=\sum_{\gamma}(f*\Phi)(v\gamma)\otimes m(v)\gamma\Psi(\gamma^{-1}w)\\ 
&=\sum_{\gamma,\varepsilon}f(\varepsilon)\varepsilon\Phi(\varepsilon^{-1}v\gamma)\otimes m(v)\gamma\Psi(\gamma^{-1}w)\\ &=\sum_{\varepsilon}f(\varepsilon)\varepsilon\alpha(\Phi\otimes \Psi)[\varepsilon^{-1}v,w]=f*\alpha(\Phi\otimes \Psi)[v,w]
\end{align*}
\begin{align*}
\alpha(\Phi\otimes\Psi*f)[v,w]&=\sum_{\gamma}\Phi(v\gamma)\otimes m(v)\gamma(\Psi*f)(\gamma^{-1}w)\\
&=\sum_{\gamma,\varepsilon}\Phi(v\gamma)\otimes m(v)\gamma(\Psi(\gamma^{-1}w\varepsilon)m(\gamma^{-1}w\varepsilon)f(\varepsilon^{-1}))\\
&=\sum_{\gamma,\varepsilon}\Phi(v\gamma)\otimes m(v)\gamma\Psi(\gamma^{-1}w\varepsilon)m(w\varepsilon)f(\varepsilon^{-1})\\
&=\sum_{\varepsilon}\alpha(\Phi\otimes\Psi)[v,w\varepsilon]m(w\varepsilon)f(\varepsilon^{-1})
\\ &= \alpha(\Phi\otimes\Psi)*f[v,w].
\end{align*}
Lastly, to see that $\alpha$ has dense range, denote by $\delta_{v}:V\to \C$ the indicator function at the element $v\in V$. The functions
\[\chi^{[v,w]}_{x\otimes y}(v',w'):=\delta_{v}(v')\delta_{w}(w')x\otimes y,\]
with $v\in V$, $w\in W$, $x\in X$ and $y\in Y$ span a dense right $C_{c}(\Gamma,C)$-submodule. Now set
\[e_{x}^{v}(v'):=\delta_{v}(v')x,\quad f^{(v,w)}_{y}(w'):=\delta_{w}(w') m(v)^{-1}(y).\]
Then it is easily verified that  $\alpha(e^{v}_{i}\otimes f^{(v,w)}_{y})=\chi^{[v,w]}_{x\otimes y}$, so $\alpha$ has dense range.
This proves the proposition.
\end{proof}
\begin{theorem}\label{Heckemod} For any $g,h\in C_{G}(\Gamma)$ there is a unitary isomorphism of bimodules
$$T^{\Gamma}_{g}\otimes_{B\rtimes_{r}\Gamma} T_{h}^{\Gamma}\xrightarrow{\sim}\bigoplus_{k=1}^{K}\left(T_{(g_{i(k)}h_{j(k)})^{-1}}^{\Gamma}\right)^{\oplus m_{k}}.$$
Consequently, for any 
 $S$-$C^{*}$-algebra $B$, the map $T:[\Gamma g^{-1}\Gamma]\mapsto [T_{g}^{\Gamma}]$ extends to a ring homomorphism
\[T:\mathbf{Z}[\Gamma, S]\to KK_{0}(B\rtimes_{r}\Gamma, B\rtimes_{r}\Gamma). \]

\end{theorem}
\begin{proof} 
By Lemma \ref{cosetHecke}, the modules $T^{\Gamma}_{g}$ and $T^{\Gamma}_{h}$ are unitarily isomorphic to those associated to the anchored bi-$\Gamma$-sets $\Gamma g^{-1}\Gamma$ and $\Gamma h^{-1}\Gamma$. By Proposition \ref{GodOfTheSun}, their tensor product is given by
\[C_{c}(\Gamma g^{-1}\Gamma, B)\otimes^{\alg}_{C_{c}(\Gamma, B)}C_{c}(\Gamma h^{-1}\Gamma ,B)\xrightarrow{\sim} C_{c}(\Gamma g^{-1}\Gamma\times_{\Gamma}\Gamma h^{-1}\Gamma, B\otimes_{B}B).\]
Since $B\otimes_{B}B\simeq B$ as $S$-modules and by Lemma \ref{bieq} there is an isomorphism of anchored bi-$\Gamma$-sets
\[\Gamma g^{-1}\Gamma\times_{\Gamma}\Gamma h^{-1}\Gamma \simeq \bigsqcup_{k=1}^{K}\bigsqcup_{\ell=1}^{m_{k}}\Gamma z_{(k,\ell)}\Gamma.\]
Taking completions we obtain the unitary bimodule isomorphism
$$T^{\Gamma}_{g}\otimes_{B\rtimes_{r}\Gamma} T_{h}^{\Gamma}\xrightarrow{\sim}\bigoplus_{k=1}^{K}\bigoplus_{\ell=1}^{m_{k}}\left(T_{z_{(k,\ell)}^{-1}}^{\Gamma}\right).$$
The definition of addition in $KK$-theory then yields
\begin{align*}T[\Gamma g^{-1}\Gamma]\otimes T[\Gamma h^{-1} \Gamma]&=[T^{\Gamma}_{g}]\otimes[T_{h}^{\Gamma}]=\sum_{k=1}^{K}\sum_{\ell=1}^{m_{k}}[T_{z_{(k,\ell)}^{-1}}^{\Gamma}]=\sum_{k=1}^{K}\sum_{\ell=1}^{m_{k}}T[\Gamma z_{(k,\ell)}\Gamma]\\
&=\sum_{k=1}^{K}m_{k}T[\Gamma z_{k}\Gamma]=T\left(\sum_{k=1}^{K}m_{k}[\Gamma z_{k}\Gamma]\right)=T([\Gamma g^{-1}\Gamma]\cdot [\Gamma h^{-1} \Gamma]),\end{align*}
showing that $[\Gamma g^{-1}\Gamma]\mapsto [T^{\Gamma}_{g}]$ is a ring homomorphism.\end{proof}
%$\mathbf{Z}[T_{g}^{\Gamma},S]$
We define $\mathcal{H}_{B}(\Gamma, S)$ to be the subring of $KK_{0}(B\rtimes_{r}\Gamma, B\rtimes_{r}\Gamma)$ generated by $T_{g}^{\Gamma}$ for $g\in C_{G}(\Gamma)$. We obviously have
\begin{corollary} If $\mathbf{Z}[\Gamma,S]$ is commutative, then $\mathcal{H}_{\Gamma}(B, S)$ is commutative.
\end{corollary}
Similarly write $\mathcal{H}_{M}(S)$ for the subring of $KK_{0}(C_{0}(M),C_{0}(M))$ generated by the classes of the correspondences $M\xleftarrow{s} M_{g}\xrightarrow{t} M$ with $g\in S$. 
\begin{corollary}\label{corrcor} Let $X$ be an $S$-space on which $\Gamma$ acts freely and properly with quotient $M:=X/\Gamma$. The map $[\Gamma g^{-1}\Gamma]\mapsto [M\xleftarrow{t} M_{g}\xrightarrow{s}M]$ defines a ring homomorphism $$\mathbf{Z}[\Gamma,S]\to KK_{0}(C_{0}(M), C_{0}(M)).$$ In particular, the double-coset product $[\Gamma g^{-1} \Gamma]\cdot [\Gamma h^{-1}\Gamma]$ corresponds to the class of the composition of correspondences $[M\xleftarrow{s_{g}} M_{g}\, { }_{t_{g}}\times_{s_{h}} M_{h}\xrightarrow{t_{h}}M]$ and there is an isomorphism $\mathcal{H}_{M}(S)\simeq \mathcal{H}_{\Gamma}(C_{0}(X), S)$.
\end{corollary}
\begin{proof} By \cite[Proposition 3.8]{mesland-sengun} the Morita equivalence isomorphism $$KK_{0}(C_{0}(X)\rtimes\Gamma, C_{0}(X)\rtimes\Gamma)\to KK_{0}(C_{0}(M), C_{0}(M))$$ maps $T^{\Gamma}_{g}$ to $T^{M}_{g}=[M\xleftarrow{s_{g}} M_{g} \xrightarrow{t_{g}} M]$. Thus the above map is the composition
\[\mathbf{Z}[\Gamma, S]\to KK_{0}(C_{0}(X)\rtimes\Gamma, C_{0}(X)\rtimes\Gamma)\to KK_{0}(C_{0}(M),C_{0}(M)),\]
whence a homomorphism. The last statement follows from \cite[Theorem 3.2]{connes-skandalis}. Clearly $\mathcal{H}_{M}(S)\simeq \mathcal{H}_{\Gamma}(C_{0}(X), S)$ under this isomorphism.
\end{proof}

\begin{remark} Corollary \ref{corrcor} is the $KK$-theoretic analogue of the well-known fact (see \cite[Chapter 7]{Shimura}) that the double-coset Hecke ring can be interpreted in terms of (topological) correspondences where the double-coset multiplication simply becomes composition of correspondences.
\end{remark}

%%%%%%%%%%%%%%%%%%%%%%%%%%%%%%%%%%
\section{Hecke equivariant exact sequences}\label{exact}
As before, let $S$ be a group such that $\Gamma \subset S \subset C_G(\Gamma)$. In this section we prove the following general result. For $S$-algebras $A$ and $B$, and any element $[x]\in KK_{i}^{S}(A,B)$ we have that
\[[T_{g}^{A\rtimes_{r}\Gamma}]\otimes j_{\Gamma} ([x])=j_{\Gamma}([x])\otimes[T_{g}^{B\rtimes_{r}\Gamma}]\in KK_{i}(A\rtimes_{r}\Gamma, B\rtimes_{r}\Gamma).\]
Here $j_{\Gamma}$ denotes the Kasparov descent map (see \cite{Kas2, kasparovconspectus})
\[j_{\Gamma}: KK^{S}_{*}(A,B)\to KK^{\Gamma}_{*}(A,B)\to KK_{*}(A\rtimes_{r}\Gamma, B\rtimes_{r}\Gamma),\]
and we have written $T^{A\rtimes_{r}\Gamma}_{g}$ for $T^{\Gamma}_{g}$ to emphasize the change of coefficient algebra. This result implies that for any $S$-equivariant semisplit extension 
\[0\to I\to A\to B\to 0,\]
of $C^{*}$-algebras that is $\Gamma$-exact in the sense  that
\[0\to I\rtimes_{r} \Gamma\to A\rtimes_{r}\Gamma \to B\rtimes_{r}\Gamma\to 0,\] is exact, the long exact sequences in both variables of the $KK$-bifunctor are Hecke equivariant. In particular we obtain Hecke equivariant exact sequences in $K$-theory and $K$-homology for various compactifications associated with locally symmetric spaces.
\subsection{The descent theorem}
Kasparov's descent construction associates to a $\Gamma$-equivariant $C^{*}$-$B$-module $X$ a $C^{*}$-module $X\rtimes_{r}\Gamma$ over $B\rtimes_{r}\Gamma$ (see \cite{Kas1,Kas2, kasparovconspectus}). To an $S$-equivariant $C^{*}$-module $X$ and a double coset $\Gamma g^{-1}\Gamma$, with $g\in S$, we associate the $(C_{c}(\Gamma, A),C_{c}(\Gamma, B))$-bimodule
\[C_{c}(\Gamma g^{-1}\Gamma, X)=\mathbf{C}[\Gamma g^{-1}\Gamma]\otimes^{\textnormal{alg}}_{\mathbf{C}}X.\] 
%The convolution product makes $C_{c}(\Gamma g^{-1} \Gamma, X)$ into a right $C_{c}(\Gamma, B)$  and a left $C_{c}(\Gamma,\End^{*}_{B}(X))$ module via the usual formulae:
%\[f * \Psi(g_{i}\xi)=\sum_{\gamma\in\Gamma}f(\gamma)\cdot\gamma(\Psi(\gamma^{-1}g_{i}\xi)), \quad \Psi * f(g_{i}\xi):=\sum _{\gamma\in\Gamma} \Psi(g_{i}\gamma)\cdot g_{i}\gamma f(\gamma^{-1}\xi).\]
%Moreover we define the inner product
%\begin{equation}\label{cosetrepinnerprod2}
%\langle\Phi,\Psi\rangle(\delta):=\sum_{i}\sum_{\gamma\in\Gamma}\gamma g_{i}^{-1}\langle\Phi(g_{i}\gamma^{-1}),\Psi(g_{i}\gamma^{-1}\delta)\rangle_{X},
%\end{equation}
cf. Section \ref{redcross}. We denote the $C^{*}$-module completion so obtained by $T_{g}^{X\rtimes_{r}\Gamma}$. %Note that a $\Gamma$-equivariant $*$-homomorphism $\pi:A\to \End^{*}_{B}(X)$ induces a $*$-homomorphism
%\[\pi:A\rtimes_{r}\Gamma\to \End^{*}_{B\rtimes_{r}\Gamma}(X\rtimes_{r}\Gamma),\]
%and given such a $\pi$ the above formula makes $T^{X\rtimes_{r}\Gamma}_{g}$ into a left $A\rtimes_{r}\Gamma$-module. 
The following Lemma is an application of Proposition \ref{GodOfTheSun}.
\begin{lemma}\label{descentlemma} Let $A$ and $B$ be $S$-$C^{*}$-algebras. Suppose that $X$ is an $S$-equivariant right $C^{*}$-module over $B$ and $\pi:A\to\End^{*}_{B}(X)$ an $S$-equivariant essential $*$-homomorphism. For every $g\in S$, there are inner product preserving bimodule homomorphisms
\begin{equation}\label{Xiso}C_{c}(\Gamma g^{-1}\Gamma, A)\otimes_{C_{c}(\Gamma,A)}^{\textnormal{alg}}C_{c}(\Gamma, X)\xrightarrow{\sim} C_{c}(\Gamma g^{-1}\Gamma, X) \xleftarrow{\sim} C_{c}(\Gamma,X)\otimes _{C_{c}(\Gamma,B)}^{\textnormal{alg}}C_{c}(\Gamma g^{-1}\Gamma, B),\end{equation}
of $(C_{c}(\Gamma,A),C_{c}(\Gamma,B))$-bimodules with dense range. Consequently the respective $C^{*}$-module completions are unitarily isomorphic $(A\rtimes_{r}\Gamma, B\rtimes_{r}\Gamma)$-bimodules.
\end{lemma}
From the identifications
\[\Gamma g^{-1}\Gamma\times_{\Gamma}\Gamma\simeq \Gamma g^{-1}\Gamma\simeq \Gamma\times_{\Gamma}\Gamma g^{-1}\Gamma,\]
given by the multiplication maps and the $S$-equivariant isomorphisms
\[X\simeq A\otimes_{A}X\simeq X\otimes_{B}B,\]
coming from the bimodule structure we obtain the explicit from of the isomorphisms in \eqref{Xiso}:
\begin{align*}
\alpha:C_{c}(\Gamma g^{-1}\Gamma, A)\otimes_{C_{c}(\Gamma,A)}^{\textnormal{alg}}C_{c}(\Gamma, X) & \to C_{c}(\Gamma g^{-1}\Gamma, X) \\
\alpha(\Psi\otimes\Phi)(\xi) & :=\sum_{\gamma\in\Gamma}\Psi(\xi\gamma)\cdot \xi\gamma\Phi(\gamma^{-1}) \\
& \\
\beta: C_{c}(\Gamma,X)\otimes _{C_{c}(\Gamma,B)}^{\textnormal{alg}}C_{c}(\Gamma g^{-1}\Gamma, B)&\to C_{c}(\Gamma g^{-1}\Gamma, X) \\
\beta(\Phi\otimes\Psi)(\xi)&:=\sum_{\gamma\in\Gamma}\Phi(\gamma)\cdot \gamma\Psi(\gamma^{-1}\xi).
\end{align*}
%\begin{proof} We will show that both bimodules in \eqref{Xiso}
%are unitarily isomorphic to the bimodule $C_{c}(\Gamma g^{-1}\Gamma, X)$. 
%\end{proof}
As before the elements $g_{i}$ are such that $\Gamma g^{-1}\Gamma=\bigsqcup_{i=1}^{d}g_{i}\Gamma$. We construct from them the following operators. 
\begin{lemma} The operator 
\begin{align*}v_{i}:C_{c}(\Gamma,X)&\to C_{c}(g_{i}\Gamma, X)\subset C_{c}(\Gamma g^{-1} \Gamma, X),\quad
(v_{i}\Phi)(g_{i}\xi):=g_{i}\Phi(\xi),
\end{align*} 
extends to an adjointable isometry $X\rtimes_{r}\Gamma\to T^{X\rtimes\Gamma}_{g}$ with adjoint given by
\[(v_{i})^{*}\Psi(\xi):=g_{i}^{-1}\Psi(g_{i}\xi).\] 
\end{lemma}
\begin{proof} The formula for the adjoint is easily verified. It follows that $(v_{i})^{*}v_{i}=1$ on $C_{c}(\Gamma, X)$, so $v_{i}$ is isometric. The composition $v_{i}v_{i}^{*}=p_{i}$, the projection onto $C_{c}(g_{i}\Gamma, X)$, which is bounded as well.
\end{proof}

\begin{theorem}\label{KKequiv} Let $(X,D)$ be an $S$-equivariant left-essential unbounded Kasparov module of parity $j$ and let $g\in S$. Then we have an equality
\[j_{\Gamma}([(X,D)])\otimes [T_{g}]=[T_{g}]\otimes j_{\Gamma}([(X,D)])\in KK_{j}(A\rtimes_{r}\Gamma,B\rtimes_{r}\Gamma).\]
\end{theorem}
\begin{proof} By Lemma \ref{descentlemma} we have bimodule isomorphisms
\[(X\rtimes_{r}\Gamma)\otimes_{B\rtimes_{r}\Gamma}T^{B\rtimes_{r}\Gamma}_{g}\xrightarrow{\beta} T^{X\rtimes_{r}\Gamma}_{g}\xleftarrow{\alpha} T^{A\rtimes_{r}\Gamma}_{g}\otimes_{A\rtimes_{r}\Gamma} (X\rtimes_{r}\Gamma).\]
Define an operator $\widehat{D}$ on the dense submodule
$$C_{c}(\Gamma g^{-1}\Gamma, \textnormal{Dom } D)\subset T^{X\rtimes_{r}\Gamma}_{g},$$ via
\[(\widehat{D}\Upsilon )(\xi):=D(\Upsilon(\xi)).\]
Then $\widehat{D}\beta=\beta(D\otimes 1)$ and hence $\widehat{D}$ is essentially self-adjoint and regular, and has locally compact resolvent. We wish to show that $\hat{D}$ represents the Kasparov product of $T^{A\rtimes_{r}\Gamma}_{g}$ and $(X,D)$, under the isomorphism $\alpha$. To this end we need to verify conditions 1-3 of \cite[Theorem 13]{Kuc}. Because the module $T^{A\rtimes_{r}\Gamma}_{g}$ carries the zero operator, only the connection condition 1 needs argument.

Let $\mathcal{A}$ denote the dense subalgebra of $A$ for such that $[D,a]$ is bounded for $a\in\mathcal{A}$. Then, for $\Psi\in C_{c}(\Gamma g^{-1} \Gamma,\mathcal{A})$, $\xi\in\Gamma$ and a fixed element $g_{i}$ we have
\begin{align*}\widehat{D}\alpha  (\Psi \otimes \Phi)(g_{i}\xi)-\alpha(&\Psi\otimes D\Phi)(g_{i}\xi)\\&=\sum_{\gamma\in\Gamma}D\Psi(g_{i}\gamma)\cdot g_{i}\gamma\Phi(\gamma^{-1}\xi)-\Psi(g_{i}\gamma)g_{i}\gamma D\Phi(\gamma^{-1}\xi)
\\
%&=\sum_{\gamma\in\Gamma}[D,\Psi(g_{i}\gamma)]g_{i}\gamma\Phi(\gamma^{-1}\xi)+\Psi(g_{i}\gamma)[D,g_{i}\gamma] \Phi(\gamma^{-1}\xi)\\
&=\sum_{\gamma\in\Gamma}([D,\Psi(g_{i}\gamma)]-\Psi(g_{i}\gamma)(D-g_{i}\gamma D \gamma^{-1}g_{i}^{-1}) )g_{i}\gamma\Phi(\gamma^{-1}\xi)\\
&=g_{i}\left(\sum_{\gamma\in\Gamma}g_{i}^{-1}([D,\Psi(g_{i}\gamma)]-\Psi(g_{i}\gamma)(D-g_{i}\gamma D \gamma^{-1}g_{i}^{-1}) \gamma\Phi(\gamma^{-1}\xi)\right)\\
&=v_{i}(C^{i}_{\Psi}* \Phi )(g_{i}\xi).
\end{align*}
Here $C^{i}_{\Psi}$ denotes the map
\begin{align*}C^{i}_{\Psi}:\Gamma &\to \End^{*}_{B}(X) \\
\gamma &\mapsto g_{i}^{-1}([D,\Psi(g_{i}\gamma)]-\Psi(g_{i}\gamma)(D-g_{i}\gamma D \gamma^{-1}g_{i}^{-1})),
\end{align*}
which is of finite support since $\Psi$ is. Such maps define adjointable operators on $C_{c}(\Gamma,X)$ via the convolution action. Writing $|\Psi\rangle:\Phi\to \Psi\otimes \Phi$ we have
\[\widehat{D}\alpha|\Psi\rangle-\alpha|\Psi\rangle D=\sum_{i=1}^{d}  v^{i}\circ C^{i}_{\Psi}:X\rtimes_{r}\Gamma\to T^{X\rtimes_{r}\Gamma}_{g}.\]
which defines a bounded adjointable operator. Thus $\widehat{D}$ satisfies Kucerovsky's connection condition as desired. 
%\todo[inline]{H: I don't think we mentioned this condition before and did not say that we desired it either.}
\end{proof}

\begin{corollary} For any $\alpha\in KK^{S}_{j}(A,B)$ and any separable $C^{*}$-algebra $C$, the induced maps
\[\alpha_{*}:KK_{i}(C, A\rtimes_{r} \Gamma)\to KK_{i+j}(C, B\rtimes_{r} \Gamma),\quad \alpha^{*}:KK_{i}(B\rtimes_{r} \Gamma,C)\to KK_{i+j}(A\rtimes_{r} \Gamma,C), \]
are Hecke equivariant. In fact we can replace $KK(C, -)$ and $KK(-,C)$ by any co- resp. contravariant functor which is homotopy invariant, split exact and stable.
\end{corollary}

\subsection{Extensions and Hecke equivariant exact sequences} The paper \cite{Thomsen} establishes, for any locally compact group $G$, an isomorphism
\[KK^{G}_{1}(A,B)\xrightarrow{\sim} \textnormal{Ext}^{G}(A\otimes\mathbb{K}_{G}, B\otimes\mathbb{K}_{G}),\]
where $\mathbb{K}_{G}\simeq \mathbb{K}(L^{2}(G\times\mathbf{N}))$. A $G$-equivariant semi-split extension
\begin{equation*} 0\to B\to E\to A\to 0,\end{equation*}
induces a $G$-equivariant semi-split extension
\[0\to B\otimes\mathbb{K}_{G}\to E\otimes\mathbb{K}_{G}\to A\otimes\mathbb{K}_{G}\to 0,\]
and thus an element in $KK^{G}_{1}(A,B)$. %We obtain long exact sequences 
%\[\cdots \to KK^{G}_{i}(C,B)\to KK^{G}_{i}(C,E)\to KK^{G}_{i}(C,A)\to\cdots,\]
%\[\cdots \to KK^{G}_{i}(A,C)\to KK^{G}_{i}(E,C)\to KK^{G}_{i}(B,C)\to\cdots,\]
%in equivariant $KK$-theory.

\begin{theorem} Let $G$ be a locally compact group, $\Gamma\subset G$ a discrete subgroup, $C_{G}(\Gamma)\subset G$ its commensurator and $S$ a group with $\Gamma\subset S\subset C_{G}(\Gamma)$. For any $\Gamma$-exact and $S$-equivariant semi-split extension
\[0\to B\to E\to A\to 0,\]
of seperable $S$-algebras and any separable $C^{*}$-algebra $C$, the exact sequences 
\begin{equation}\label{cov}\cdots \to KK_{i}(C,B\rtimes_{r}\Gamma)\to KK_{i}(C,E\rtimes_{r}\Gamma)\to KK_{i}(C,A\rtimes_{r}\Gamma)\to\cdots\end{equation}
\begin{equation}\label{contrav}\cdots \to KK_{i}(A\rtimes_{r}\Gamma,C)\to KK_{i}(E\rtimes_{r}\Gamma,C)\to KK_{i}(B\rtimes_{r}\Gamma,C)\to\cdots,\end{equation}
are $\mathbf{Z}[\Gamma, S]$-equivariant.
\end{theorem}

\begin{proof} Exactness of $\Gamma$ implies that we obtain a semi-split extension
\begin{equation}\label{redext}0\to B\rtimes_{r} \Gamma \to E\rtimes_{r}\Gamma \to A\rtimes_{r}\Gamma \to 0,\end{equation}
yielding the exact sequences \eqref{cov} and \eqref{contrav}. By Theorem \ref{Heckemod} all groups in these exact sequences are Hecke modules. In sequence \eqref{cov}, the maps
\[KK_{i}(C,B\rtimes_{r}\Gamma)\to KK_{i}(C,E\rtimes_{r}\Gamma),\quad KK_{i}(C,E\rtimes_{r}\Gamma)\to KK_{i}(C,A\rtimes_{r}\Gamma),\]
are induced by elements in $KK_{0}(B\rtimes_{r}\Gamma,E\rtimes_{r}\Gamma)$ and $KK_{0}(A\rtimes_{r}\Gamma,E\rtimes_{r}\Gamma)$, respectively. These elements are in the image of the descent maps
\[KK_{0}^{S}(B,E)\to KK_{0}^{\Gamma}(B,E)\to KK_{0}(B\rtimes_{r}\Gamma,E\rtimes_{r}\Gamma),\]
\[KK_{0}^{S}(E,A)\to KK_{0}^{\Gamma}(E,A)\to KK_{0}(E\rtimes_{r}\Gamma,A\rtimes_{r}\Gamma),\]
and thus are Hecke equivariant by Theorem \ref{KKequiv}. Since the extension \eqref{redext} is semi-split it defines a class $[\textnormal{Ext}]\in KK_{1}^{S}(A,B)$. The boundary maps in the exact sequence \eqref{cov} are implemented by an element $\partial\in KK_{1}(A\rtimes_{r}\Gamma, B\rtimes_{r}\Gamma)$, and this element is the image of $[\textnormal{Ext}]$ under the composition
\[KK_{1}^{S}(A,B)\to KK_{1}^{\Gamma}(A,B)\to KK_{1}(A\rtimes_{r}\Gamma, B\rtimes_{r}\Gamma).\]
Thus by Theorem \ref{KKequiv} the boundary maps in the sequence \eqref{cov} are Hecke equivariant. The argument for sequence \eqref{contrav} is similar.
\end{proof}

Interesting examples of $S$-equivariant extensions come from partial compactifications of $G$-spaces. Let $X$ be a locally compact space with a $G$-action. A partial $S$-compactification is a $S$-space $\overline{X}$ which contains $X$ as an open dense subset. We write $\partial X:=\overline{X}\setminus X$ and we obtain the $S$-equivariant exact sequence
\[0\to C_{0}(X)\to C_{0}(\overline{X}) \to C_{0}(\partial X)\to 0.\]

\begin{example} Let $G=\textnormal{Isom}(\mathbf{H})$ where $\mathbf{H}$ is the real hyperbolic $n$-space. The geodesic compactification $\overline{\mathbf{H}}$ of $\mathbf{H}$ is a $G$-compactification and thus, it is an $S$-compactification for any lattice $\Gamma \subset G$ and subgroup $\Gamma\subset S\subset C_{G}(\Gamma)$. The associated Hecke equivariant exact sequence in $K$-homology has been studied extensively in \cite{mesland-sengun}. For torsion free $\Gamma$ and $M:=X/\Gamma$, there is a Morita equivalence $C_{0}(M)\sim C_{0}(X)\rtimes_{r}\Gamma$, and a $KK$-equivalence $C(\overline{\mathbf{H}})\rtimes_{r}\Gamma\sim C_{r}^{*}(\Gamma)$. The exact sequence takes the form
\[\cdots\to K_{*}(C_{0}(M))\to K_{*}(C^{*}_{r}(\Gamma) )\to K_{*}(C(\partial\mathbf{H})\rtimes_{r}\Gamma)\to\cdots ,\]
as in \cite{EM, emerson-nica}.
\end{example}

\begin{example}
Let $G$ be the group of real points of a reductive algebraic group ${\bf G}$ over $\Q$ and let $X$ be its associated global symmetric space.  The Borel-Serre partial compactification $\overline{X}$ of $X$ is a ${\bf G}(\Q)$-compactification but not a $G$-compactification (see \cite{borel-serre}). However if $\Gamma \subset {\bf G}(\Q)$ is an arithmetic subgroup, then $C_G(\Gamma) = {\bf G}(\Q)$. So $\overline{X}$ is a $C_G(\Gamma)$-compactification. The action of $\Gamma$ on $\overline{X}$ is cocompact and continues to be proper. Writing $M:=X/\Gamma$ for torsion free $\Gamma$ , we obtain the Borel-Serre compactification $\overline{M}:=\overline{X}/\Gamma$ of $M$ and its boundary $\partial \overline{M}:=\partial X/\Gamma$. There are Morita equivalences 
$$C_{0}(X)\rtimes_{r}\Gamma\sim C_{0}(M),\quad C_{0}(\overline{X})\rtimes_{r}\Gamma\sim C_{0}(\overline{M}),\quad C_0(\partial X)\rtimes_{r}\Gamma\sim C_{0}(\partial \overline{M}).$$
The exact sequence thus reduces to the topological $K$-theory sequence
\[\cdots\to K^{*}(M)\to K^{*}(\overline{M})\to K^{*}(\partial \overline{M})\to\cdots\]
\end{example}
\noindent of the pair $(\overline{M},\partial \overline{M})$.

%%%%%%%%%%%%%%%%%%%%%%%%%%%%%%%%%%%%%%%%%%%%%

\end{document}